\documentclass[10pt]{article}
\usepackage{amsmath,amssymb,amsthm}
\usepackage{epsfig}
\usepackage{color}

\title{On the $k$-limited packing numbers in graphs}

\author {
Babak Samadi\\
Department of Mathematics\\
Arak University\\
Arak, IRI\\
{\tt b-samadi@araku.ac.ir}\vspace{3mm}\\
}

\date{}

 \newtheorem{theorem}{Theorem}[section]

\newtheorem{lemma}[theorem]{Lemma}

\theoremstyle{definition}

\begin{document}

\maketitle 
\begin{abstract}
\noindent We give a sharp lower bound on the lower $k$-limited packing number of a general graph. Moreover, we establish a Nordhaus-Gaddum type bound on $2$-limited packing number of a graph $G$ of order $n$ as $L‎_{2}(G)+L‎_{2}(\bar G)‎‎\leq n+2‎$. Also, we investigate the concepts of packing number ($1$-limited packing number) and open packing number in graphs with more details. In this way, by making use of the classic result of Meir and Moon (1975) and its total version (2005) we prove the new upper bound $‎\gamma(T)‎\leq(n-\ell+2s)/3‎‎$ for every tree $T$ of order $n$ with $\ell$ leaves and $s$ support vertices and improve $‎\gamma‎_{t}(T)‎‎‎\leq (n+s)/2‎$, that was first proved by Chellali and Haynes in 2004. 
\end{abstract}
\section{Introduction}

Throughout this paper, let $G$ be a finite connected graph with vertex set $V=V(G)$, edge set $E=E(G)$, minimum degree $\delta=\delta(G)$ and maximum degree $\Delta=\Delta(G)$. Recall that a {\em pendant vertex} of $G$ (a {\em leaf} of a tree $T$) is a vertex of degree $1$ and a support vertex is a vertex having at least one pendant vertex in its neighbothood. We use \cite{we} as a reference for terminology and notation which are not defined here. For any vertex $v\in V$, $N(v)=\{u\in G\mid uv\in E(G)\}$ denotes the {\em open neighborhood} of $v$ in $G$, and $N[v]=N(v)\cup \{v\}$ denotes its {\em closed neighborhood}. For a set $A$ of vertices, its open neighborhood is the set $N(A)=‎\cup‎_{v\in A}N(v)‎‎$. The complement $\bar G$ of a graph $G$ has vertex set $V$ and $uv‎\in E(\bar G)‎$ if and only if $uv‎\notin E(G)$. For any graph parameter $‎\psi‎$, bounds on $‎\psi(G)+‎\psi(\bar G)‎‎$ and $\psi(G)\psi(\bar G)‎‎$ are called Nordhaus-Gaddum inequalities. For more information about this subject the reader can consult \cite{ah}.\\
A set  $S\subseteq V$ is a {\em dominating set (total dominating set)} in $G$ if each vertex in $V \setminus S$ (in $V$) is adjacent to at least one vertex in $S$. The {\em domination number $\gamma(G)$ (total domination number $\gamma_{t}(G)$)} is the minimum cardinality of a dominating set (total dominating set) in $G$. A subset $B \subseteq V$ is a {\em packing (open packing)} in $G$ if for every distinct vertices $u,v‎\in B‎$, $N[u]‎‎\cap N[v]‎=‎\emptyset‎$ ($N(u)‎‎\cap N(v)=‎\emptyset‎‎‎$). The {\em packing number (open packing number)} $\rho(G)$ ($‎\rho‎‎‎^{o}‎(G)‎‎$) is the maximum cardinality of a packing (an open packing) in $G$. Also, the {\em lower packing number}, denoted $‎\rho‎_{L‎}‎‎(G)$, is the minimum cardinality of a maximal packing in $G$.\\
Clearly, $B \subseteq V$ is a packing (an open packing) in $G$ if and only if $|N[v]‎\cap B‎|‎\leq1‎$ ($|N(v)‎\cap B‎|‎\leq1‎$), for all $v‎\in V‎$. Here, we prefer to work with these definitions on these parameters rather than the previous ones.\\
Gallant et al. \cite{gghr} introduced the concept of limited packing in graphs and exhibited some real-world applications of it to network security, market saturation and codes (Also, the authors in \cite{hks} presented some results as an application of the concept of the limited packing). A subset $B \subseteq V$ is a {\em $k$-limited packing} in $G$ if $|N[u]‎\cap B‎|‎\leq k‎‎$, for every vertex $u\in V$. The {\em $k$-limited packing number} $L‎_{k}‎(G)$ is the maximum cardinality of a $k$-limited packing in $G$. Also, the {\em lower $k$-limited packing number}, denoted $L‎^{\ell}_{k‎}‎‎(G)‎‎‎‎‎$, is the minimum cardinality of a maximal $k$-limited packing in $G$. Obviously, $L‎^{\ell}_{k‎}‎‎(G)‎\leq‎‎‎‎‎‎ L‎_{k}‎(G)$. These two concepts generalize the concepts of packing and lower packing, respectively.\\
We prove a sharp lower bound on lower $k$-limited packing number of a general graph $G$ and drive the lower bound $n/2$ for lower $2$-limited packing number of a cubic graph of order $n$ that is stronger than its similar result in \cite{gghr}. Also, we show that $L‎_{2}(G)+L‎_{2}(\bar G)‎‎\leq n+2‎$ for each graph $G$ of order $n$.\\
Meir and Moon \cite{mm} proved that $‎\rho(T)=‎\gamma(T)‎‎$, for every tree $T$. We prove an upper bound on domination number of a graph and use this classic result to show that $(n-\ell+2s)/3$ is an upper bound on domination number of a tree of order $n$ with $\ell$ leaves and $s$ support vertices. Finally we improve $‎\gamma‎_{t}(T)‎\leq (n+s)/2‎‎‎$, that was first proved by Chellali and Haynes \cite{ch}, as an immediate result.


\section{Results on the $k$-limited packing number}

Gallant et al. \cite{gghr} exhibited the lower bound $L‎_{2}(G)‎‎\geq ‎n/4‎‎$, for a cubic graph $G$ of order $n$. In this section we establish a tight lower bound on lower $k$-limited packing number of a general graph, also we observe that it slightly improves the lower bound in \cite{gghr} as the special case $k=2$. First, we need the following useful lemma that is an immediate result of the definition of a maximal $k$-limited packing.
\begin{lemma}
Let $G$ be a graph and $k$ be a positive integer. Then a $k$-limited packing set in $G$ is maximal if and only if every vertex in $V‎\setminus B‎$ belongs to the closed neighborhood of a vertex $u$ with $|N[u]‎\cap B‎|=k$.
\end{lemma}
We now construct a $r$-regular graph $G‎_{k,r}‎=(V,E)$, for $k‎\leq r‎$, of order $n$. We make use of it to show that the lower bound in Theorem 2.2 is sharp. Consider the partition $V‎_{1}‎$, $V‎_{2}‎$ and $V‎_{3}‎$ of $V$ such that ‎$G‎_{k,r}‎[V‎_{1}]$ is a $(r-1)$-regular graph of which every vertex has one neighbor in $V‎_{2}‎$. Every vertex in $V‎_{2}‎$ has $r-k$ and $k$ neighbors in $V‎_{1}‎$ and $V‎_{3}‎$, respectively. Also, every vertex in $V‎_{3}‎$ has $r$ neighbors in $V‎_{2}‎$.\vspace{1.5mm}\\
We are now in a position to present the main theorem of this section.
\begin{theorem}
If $G$ is graph of order $n$ and $k$ is a positive integer, then
$$L‎^{\ell}_{k‎}‎‎(G)‎\geq ‎\dfrac{kn}{‎\Delta(‎\Delta-k+1‎)‎+k}‎$$
and this bound is sharp.
\end{theorem}
\begin{proof}
Since $L‎^{\ell}_{k‎}‎‎(G)‎=n‎$ for $k‎\geq ‎\Delta+1‎‎$, we may assume that $k‎\leq ‎\Delta‎‎$. Let $B$ be a maximal $k$-limited packing set in $G$ such that $|B|=L‎^{\ell}_{k‎}‎‎(G)‎$. For all $i‎\leq k‎\leq ‎\Delta‎$, we define
$$B‎_{i}‎=\{u‎\in B‎||N[u]‎\cap B‎|=i\}\ \ \& \ \ B'_{i}‎=\{u‎\in V‎\setminus B‎‎||N[u]‎\cap B‎|=i\}.$$
Obviously, $|B|=\sum_{i=1}^{k}|B‎_{i}‎|$ and $|V‎\setminus B‎|=\sum_{i=0}^{k}|B'_{i}‎|$.\\
By Lemma 2.1, every vertex in $B'_{0}$ has at least one neighbor in $B'_{k}$. On the other hand, every vertex in $B'_{k}$ has at most $‎\Delta‎-k$ neighbors in $B'_{0}$. It follows that
$$|B'_{0}|‎\leq|[B'_{0},B'_{k}]|‎‎\leq(‎\Delta‎-k)|B'_{k}|‎$$
and hence
\begin{equation}
|B'_{0}|-(‎\Delta‎-k)|B'_{k}|‎\leq0.‎
\end{equation}
Lemma 2.1 implies that every vertex in $‎X=\cup‎_{i=0}^{k-1}B'_{i}$ has at least one neighbor in $B‎_{k}$ or $B'_{k}$. On the other hand, every vertex in $B'_{k}$ and $B‎_{k}$ has at most $‎\Delta‎-k$ and $‎\Delta‎-k+1$ neighbors in $X$, respectively. This leads to
$$\sum_{i=0}^{k-1}|B'_{i}|‎\leq‎|[X,B'_{k}]|+|[X,B_{k}]|‎\leq(‎\Delta‎-k)|B'_{k}|‎+(‎\Delta‎-k+1)|B_{k}|$$
and therefore
\begin{equation}
\sum_{i=0}^{k-1}|B'_{i}|-(‎\Delta‎-k)|B'_{k}|‎‎\leq (‎\Delta‎-k+1)|B_{k}|.‎
\end{equation}
Furthermore, by double counting we obtain
\begin{equation}
\sum_{i=1}^{k}i|B'_{i}|=|[B,V‎\setminus B‎]|‎\leq‎\sum_{i=1}^{k}(‎\Delta‎-i+1)|B‎_{i}‎|.
\end{equation}
For convenience, we let $p$ and $q$ to be min$\{‎\frac{k}{‎\Delta‎-k+1},1‎\}$ and max$\{\frac{k}{‎\Delta‎-k+1},1\}-1$, respectively. Then $k-(p+q)(‎\Delta‎-k)=p+q$ and $q+i‎\geq q+1‎\geq p+q‎‎$, for $i‎\geq1‎$. Now we arrive at
\begin{equation}
(p+q)|V‎\setminus B‎|‎\leq (k-(p+q)(‎\Delta‎-k))‎|B'_{k}|+(p+q)|B'_{0}|+\sum_{i=1}^{k-1}(q+i)|B'_{i}|=‎\eta.‎
\end{equation}
Adding $p$ times (1), $q$ times (2) and (3) gives
\begin{equation}\label{EQ6}
\begin{array}{lcl}
‎\eta‎\leq q(‎\Delta‎-k+1)‎‎|B_{k}|&+&\sum_{i=1}^{k}(‎\Delta‎-i+1)|B‎_{i}‎|=(q+1)(‎\Delta‎-k+1)‎‎|B_{k}|\\
&+&\sum_{i=1}^{k-1}(‎\Delta‎-i+1)|B‎_{i}‎|‎=‎‎\xi‎‎‎.‎
\end{array}
\end{equation}
Since $(q+1)(‎\Delta‎-k+1)‎\leq ‎\Delta‎$, then $‎\xi‎\leq ‎\Delta‎|B|‎‎$. Now, by inequalities (4) and (5) we have
$$(p+q)|V‎\setminus B‎|‎\leq r|B|‎‎‎$$
and so $L‎^{\ell}_{k‎}‎‎(G)‎=|B|‎\geq \frac{kn}{‎\Delta‎(‎\Delta‎-k+1)+k}‎‎‎$.\vspace{1mm}\\
Now we show that the bound is sharp for all $k‎\leq ‎\Delta‎‎$. Consider the graph $G‎_{k,r}‎$. Then $|V‎_{1}‎|=|[V‎_{1},V‎_{2}‎‎]|=(r-k)|V‎_{2}‎|$ and $k|V‎_{2}‎|=|[V‎_{2},V‎_{3}‎‎]|=r|V‎_{3}‎|$. It is easy to see that there exists a positive integer $t$ such that $|V‎_{2}‎|=tr/g$ and $|V‎_{3}‎|=tk/g$, where $g$ = gcd$(k,r)$. Then $|V‎_{1}‎|=(r-k)tr/g$ and therefore $n=(r^2-kr+r+k)t/g$. According to Lemma 2.1, $V‎_{3}‎$ is a maximal $k$-limited packing in $G$. This leads to $L‎^{\ell}_{k‎}‎‎(G)‎‎\leq|V‎_{3}‎|‎=‎\frac{kn}{‎r(‎r-k+1‎)‎+k}$. This completes the proof.
\end{proof}
When we restrict the problem to the case $(k,r)=(2,3)$ then
$$L‎_{2}(G‎‎)‎‎\geq L‎^{\ell}_{2‎}‎‎(G)‎‎‎\geq n/4.$$
We note that Balister et al. \cite{bbg} proved the best-possible result that if $G$ is a cubic graph, then $L‎_{2}(G)‎\geq n/3‎‎$. This improves the lower bound given in \cite{gghr}. This shows that the lower bound $n/4$ for $2$-limited packing number of a cubic graph $G$ is not sharp, while it is sharp for lower $2$-limited packing number of $G$, by Theorem 2.2.\vspace{1mm}\\
In the next theorem we give a Nordhaus-Gaddum inequality for $2$-limited packing number of a graph just in terms of the order without any additional condition.
\begin{theorem}
Let $G$ be a graph of order $n$. Then
$$L‎_{2}(G)+L‎_{2}(\bar G)‎‎\leq n+2‎$$
and this bound is sharp.
\end{theorem}
\begin{proof}
The result is obvious for $n=1$ and if $G=K‎_{n}‎$, for $n‎\geq2‎$, then $L‎_{2}(G)+L‎_{2}(\bar G)‎‎=n+2‎$. Therefore, we may assume that $‎\Delta(G),‎\Delta(\bar G)‎\geq1‎‎$. If $‎\Delta(G)=\Delta(\bar G)‎=1$ then $n‎\leq3‎$ that implies the inequality. Hence, without loss of generality we assume that $‎\Delta(G)‎\geq2‎$.\\
Now we proceed by induction on the order $n$. We suppose it is true for $n-1$ and prove it for $n$. Let $B$ and $\bar B$ be maximum $2$-limited packings in $G$ and $\bar G$, respectively. Let $u$ be a vertex in $G$ with $deg(u)=‎\Delta(G)‎$. Then there exists a vertex $v‎\in N[u]‎$ not belonging to $B$, for otherwise $|N[u]‎\cap B‎|=‎\Delta(G)+1‎\geq3‎‎$, a contradiction. Applying the inductive hypothesis to $G-v$, we obtain
\begin{equation}
L‎_{2}(G-v)+L‎_{2}(\bar G-v)‎‎‎\leq ‎n+1.
\end{equation}
On the other hand, $B$ is a $2$-limited packing in $G-v$ and hence $|B|‎\leq L‎_{2}(G-v)‎$. Moreover, it follows by the definition that $|\bar B|-1‎\leq L‎_{2}(\bar G-v)‎$. By (6), we have 
$$|B|+|\bar B|-1‎\leq n+1.‎$$
This completes the proof.
\end{proof}


\section{The special case $k=1$}

In this section we consider the case $k=1$ in a more specific way. As a special case of Theorem 2.2 we have
$$L‎^{\ell}_{1‎}‎‎(G)=‎\rho‎_{L}(G)‎‎‎\geq n/(‎\Delta‎‎^{2}+1‎‎)‎$$
that was first presented by Henning in \cite{h} (also, see \cite{bbg} and \cite{gz}). Moreover, Gagarin and Zverovich \cite{gz} presented the lower bound
$$‎\rho(G)=L‎_{1}(G)‎‎‎\geq(n+‎\Delta(‎\Delta-‎\delta‎‎)‎)/(‎\Delta‎^{2}+1‎‎)‎$$
for a graph $G$ of order $n$. In the following we shall present a lower bound on the packing number of a graph that is tighter for graphs with support vertices. We need the following useful lemma.
\begin{lemma}
Let $G$ be a graph of order $n$ with $s$ support vertices. Let $\ell‎_{1},...,\ell‎_{s}$ be the pendant vertices adjacent to the support vertices $v‎_{1},...,v‎_{s}‎$, respectively.  Then the following statements hold. \vspace{2.5mm}\\
$(i)$\ \ There exists a maximum packing set $B$ in $G$ containing $\ell‎_{1},...,\ell‎_{s}$. \vspace{1mm}\\
$(ii)$\ \ There exists a maximum open packing set $B$ in $G$ containing $\ell‎_{1},...,\ell‎_{s}$.

\end{lemma}
\begin{proof}
We only prove $(i)$, as $(ii)$ can be proved in a similar fashion. Let $B$ be a maximum packing in $G$. Suppose that the vertices $l‎_{1}‎$ and $v‎_{1}‎$ do not belong to $B$. Since $B$ is a maximum packing in $G$ then the support vertex $v‎_{1}‎$ has a neighbor $u‎_{1}‎$ in $B$, otherwise $B‎\cup \{\ell‎_{1}‎\}‎$ is a packing in $G$ that contradicts the maximality of $B$. Now $B‎_{1}‎=(B‎\setminus \{u‎_{1}‎\})‎\cup \{\ell‎_{1}‎\}‎‎$ is a packing in $G$ and $|B|=|B‎_{1}‎|$. Now let (just) one of $\ell‎_{1}‎$ and $v‎_{1}‎$ belongs to $B$. Without loss of generality we can assume that $\ell‎_{1}\in B$, otherwise we replace $v‎_{1}‎$ with $\ell‎_{1}‎$ in $B$ and the new set would be a packing in $G$ with the same cardinality as $B$, as well. Repeating this precess for all pairs $v‎_{i}$‎ and $\ell‎_{i}‎$, for all $i‎\geq1‎$, we can construct a maximum packing set $B$ containing the pendant vertices $\ell‎_{1},...,\ell‎_{s}$. 
\end{proof}
Now we give the lower bound in terms of order, maximum degree and the number of support vertices. In fact, the bound is sharp for some graphs with no support vertices. For example, the lower bound is achieved by the Petersen graph which has 10 vertices and diameter 2. Also, by taking graphs consisting of many vertex-disjoint copies of the Petersen graph, this shows that for every $n$ divisible by 10, there is a graph on $n$ vertices with $‎\rho‎‎‎(G)=n/10=n/(‎\Delta‎^{2}‎‎+1)$. But the lower bound works better if we consider all graphs with support vertices, especially trees. 
\begin{theorem}
Let $G$ be a graph of order $n$ with $s$ support vertices and maximum degree $‎\Delta‎$. Then
$$‎\rho‎‎‎(G)‎\geq‎\dfrac{n+s(‎\Delta‎^{2}-‎\Delta‎‎‎)}{1+\Delta‎^2‎}‎‎‎‎$$
and this bound is sharp.
\end{theorem}
\begin{proof}
Let $\{v‎_{1},...,v‎_{s}‎‎\}$ be the set of support vertices and $\{\ell‎_{1},...,\ell‎_{s}\}‎‎$ be a subset of pendant vertices such that $l‎_{i}v‎_{i}‎\in E(G)‎‎‎$, for all $1‎\leq i‎\leq s‎‎$. Let $B$ be a maximum packing in $G$ with $‎\rho‎‎(G)=|B|$ containing $\ell‎_{1},...,\ell‎_{s}$, by Part $(i)$ of Lemma 3.1. \\
Let $A‎_{0}=\{v‎\in V‎\setminus B|N(v)‎\cap B=‎\emptyset‎‎‎‎\}‎$ and $A‎_{1}=(V‎\setminus B)‎\setminus A‎_{0}‎‎$. Also, $a‎_{0}=|A‎_{0}‎|‎$ and $a‎_{1}‎=|A‎_{1}‎|$. Since $B$ is a packing in $G$, then $V‎\setminus B‎$ is a disjoint union of the sets $A‎_{0}‎$ and $A‎_{1}‎$. Now we get
\begin{equation}
n=|B|+|V‎\setminus B‎|=|B|+a‎_{0}+a‎_{1}.
\end{equation}
On the other hand, if $[B,V‎\setminus B‎]$ is the set of edges having one end point in $B$ and the other in $V‎\setminus B‎$, then
\begin{equation}
|[B,V‎\setminus B‎]|=a‎_{1}‎\leq ‎\Delta(|B|-s)‎‎+s.
\end{equation}
{\bf Case 1.} Let $A‎_{0}=‎‎\emptyset‎‎‎‎$. By inequalities (7) and (8), we have
\begin{equation*}
n=|B|+a‎_{1}‎\leq |B|+‎\Delta(|B|-s)‎‎+s.‎
\end{equation*}
Therefore,
$$‎\rho‎_{L‎}‎‎(G)=|B|‎\geq ‎\frac{n+s(‎\Delta-1‎)}{1+‎\Delta‎}‎\geq \dfrac{n+s(‎\Delta‎^{2}-‎\Delta‎‎‎)}{1+\Delta‎^2‎}$$
as desired.\\
{\bf Case 2.} Let $A‎_{0}‎\neq‎‎‎\emptyset‎‎‎‎$. Let $v‎\in A‎_{0}‎$. There exists a vertex $u$ in $A‎_{1}‎$ such that $u‎\in N(v)‎$, otherwise $B‎\cup \{v\}‎$ is a packing in $G$ that contradicts the maximality of $B$. This shows that the set $Q=\{v‎\in N(A‎_{0}‎)‎||N(v)‎\cap B‎|=1\}$ is nonempty and $A‎_{0}‎‎\subseteq N(Q)‎$. Every vertex in $Q$ has at most $‎‎\Delta-1‎‎$ neighbors in $A‎_{0}‎$ and therefore $Q$ has at most $(‎‎\Delta-1)|Q|$ neighbors in $A‎_{0}‎$. Now we deduce that
$$a‎_{0}=|A‎_{0}‎|=|N(Q)‎‎\cap A‎_{0}‎‎|‎\leq(‎‎\Delta-1)|Q|‎\leq (‎‎\Delta-1)a‎_{1}.‎‎$$
By (7) and (8), we have
$$n‎\leq|B|+‎(‎‎\Delta-1)a‎_{1}+a‎_{1}=|B|+‎\Delta a‎_{1}‎\leq|B|+\Delta(‎\Delta(|B|-s)‎‎+s)‎‎.$$
Therefore $‎\rho‎_{L‎}‎‎(G)=|B|‎\geq (n+s(‎\Delta^2-‎\Delta‎‎))/(1+‎\Delta^2‎)$, as desired.\\
To show that the lower bound is sharp, we consider the star $K‎_{1,n-1}‎$ for $n‎\geq2‎$. This completes the proof.
\end{proof}
\section{Upper bounds}
We can use the packing $B$ described in Part $(i)$ of Lemma 3.1 to obtain a tight upper bound on domination number of a graph as follows:
\begin{theorem}
For any graph $G$ of order $n‎\geq3‎$, with $\ell$ pendant vertices and $s$ support vertices,
$$‎\rho(G)‎‎\leq ‎\frac{n-\ell+‎\delta's}{1+‎\delta'‎}‎‎$$
where $‎\delta'‎$ is the minimum degree taken over all vertices that are not pendant vertices.
\end{theorem}
\begin{proof}
Let $B$ be a maximum packing in $G$. According to the first part of Lemma 3.1, we may assume that $\{\ell‎_{1},...,\ell‎_{s}\}‎\subseteq B‎‎‎$. Since $B$ is a packing in $G$, then all support vertices belong to $V‎\setminus B‎$ and pendant vertices different from $\ell‎_{1},...,\ell‎_{s}$ are in $V‎\setminus B‎$, as well. Clearly, all pendant vertices in $V‎\setminus B‎$ have no neighbors in $B$. Therefore
\begin{equation}
|[B,V‎\setminus B‎]|‎\leq n-|B|‎-(\ell-s).
\end{equation}
Since $B$ is an independent set then every vertex in $B‎\setminus ‎\{\ell‎_{1},...,\ell‎_{s}\}$ has at least $‎\delta'‎$ neighbors in $V‎\setminus B‎$. Hence
\begin{equation}
|[B,V‎\setminus B‎]|‎‎\geq ‎\delta'‎(|B|-s)+s‎.
\end{equation}
Together inequalities (9) and (10) imply the desired upper bound. 
\end{proof}
Now, we bound the domination number of a tree from above by considering the number of its leaves and support vertices. We will need the following useful lemma due to Meir and Moon \cite{mm}.
\begin{lemma}\cite{mm}
If $T$ is any tree, then $‎\gamma(T)=‎\rho(T)‎‎$.
\end{lemma}
\begin{theorem}
Let $T$ be a tree of order $n$, with $\ell$ leaves and $s$ support vertices. Then 
$$‎\gamma(T)\leq ‎\frac{n-\ell+2‎s}{3}$$
and this bound is sharp.
\end{theorem}
\begin{proof}
The result is obvious for $n=1,2$. So, let $n‎\geq3‎$. Since $‎\delta'‎\geq max\{2,‎\delta‎\}=2‎‎‎‎$, then $‎‎\rho(T)‎\leq‎‎\frac{n-\ell+‎\delta'‎s}{1+‎\delta'}‎\leq ‎\frac{n-\ell+2s}{3}‎‎$. Now Lemma 4.2 implies
\begin{equation*}
‎\gamma(T)=‎‎\rho(T)‎‎\leq ‎\frac{n-\ell+2‎s}{3}.
\end{equation*}
T‎o show that the lower bound is sharp, we consider all trees $T$ in which all vertices are leaves or support vertices. In this case, it is easy to see that $\gamma(T)=‎‎\rho(T)=s=(n-\ell+2s)/3$.
\end{proof}

Using the open packing $B$ given in Part $(ii)$ of Lemma 3.1 we can bound the open packing number of a graph from above as follows:
\begin{theorem}
For any graph $G$ of order $n‎\geq3‎$ with $s$ support vertices
$$‎\rho‎^{o}‎(G)‎‎‎‎\leq ‎\frac{n+(‎\delta'-1‎)s}{‎\delta'‎}.‎‎$$
\end{theorem}
\begin{proof}
Let $B$ be maximum open packing in $G$. According to the second part of Lemma 3.1, we may assume that $\{\ell‎_{1},...,\ell‎_{s}\}‎\subseteq B‎‎‎$. After perhaps relabelling the index $1‎\leq i‎\leq s‎‎$, let $\{\ell‎_{1},...,\ell‎_{s‎_{1}‎}\}‎$ be the set of all pendant vertices in $B$ with neighbors in $V‎\setminus B‎$. This shows that $s‎_{2}=s-s‎_{1}‎‎$ pendant vertices in $B$ have their support vertices in $B$. Since $B$ is an open packing in $G$, then
\begin{equation*}
s‎_{1}‎+(‎\delta'-1‎)(|B|-s‎_{1}-s‎_{2}‎‎)‎\leq|[B,V‎\setminus B‎]|‎.
\end{equation*}
On the other hand every vertex in $V‎\setminus B‎$ has at most one neighbor in $B$. Hence
\begin{equation*}
|[B,V‎\setminus B‎]|‎\leq n-|B|‎‎.
\end{equation*}
These two inequalities imply the desired upper bound.
\end{proof}
Rall \cite{r} established an open packing and total domination equality for trees analogous to the well-known result of Meir and Moon as follows:
\begin{lemma}\cite{r}
For every tree $T$ of order at least two, $‎\gamma‎_{t}(T)=‎\rho^o(T)‎‎‎$.
\end{lemma}
Since $‎\delta'‎\geq2‎‎$ for $n‎\geq3‎$, then $\rho‎^{o}‎(T)‎‎‎‎\leq\frac{n+(‎\delta'-1‎)s}{‎\delta'‎}‎\leq \frac{n+s}{‎2‎}‎$.
Now, as an immediate result of Theorem 4.4 and Lemma 4.5 we have improved the following result, that was proved by Chellali and Haynes by induction on the order $n$. 
\begin{theorem}\cite{ch}
If $T$ is a tree of order $n‎\geq3‎$ with $s$ support vertices, then $‎\gamma‎_{t}(T)‎\leq ‎\frac{n+s}{2}‎‎‎‎$ and this bound is sharp.
\end{theorem}


‎‎‎‎‎‎‎‎‎

\end{document}